\documentclass[11pt]{amsart}
\usepackage{amsmath}
\usepackage{bbding}
\usepackage{tikz}
\usepackage{amsmath,amssymb}
\theoremstyle{plain}
\newtheorem{theorem}{Theorem}[section]
\newtheorem{lemma}[theorem]{Lemma}
\newtheorem{prop}[theorem]{Proposition}

\theoremstyle{definition}
\newtheorem{remark}[theorem]{Remark}

\numberwithin{equation}{section}

\def\be{\begin{equation}}
\def\ee{\end{equation}}

\begin{document}

\title[Mixed Local-Nonlocal Operator]
{Liouville Theorems for the Lane-Emden Equation\\ Involving a Mixed Local-Nonlocal Operator}

\author[Guo]{Yahong Guo}
\address{School of Mathematical Sciences\\
Shanghai Jiao Tong University\\
Shanghai, 200240, China} \email{yhguo@sjtu.edu.cn}
\author[Li]{Congming Li}
\address{School of Mathematical Sciences\\
Shanghai Jiao Tong University\\
Shanghai, 200240, China}  \email{congming.li@sjtu.edu.cn}
\author[Xie]{Jiongduo Xie}
\address{School of Mathematical Sciences\\
Shanghai Jiao Tong University\\
Shanghai, 200240, China}
\email{jiongduoxie@outlook.com
}

\begin{abstract}
In this paper, we investigate the existence of positive supersolutions for the following mixed local–nonlocal Lane–Emden type equation:
$$
-\Delta u+(-\Delta)^s u=u^q\quad\text{in }\mathbb R^n,  
$$
where $n\geq 3$, $s\in(0,1)$, and $q>1$. More precisely, we prove that the equation admits positive distributional supersolutions if and only if $q>\frac{n}{n-2s}$. In the process, we establish several novel properties of mixed local–nonlocal operators, including sharp asymptotic estimates for the fundamental solution, a maximum principle in the distribution sense, and an equivalent integral inequality for supersolutions.
\end{abstract}

\maketitle

\section{Introduction}\label{sec-Intro}

In this paper, we consider the following Lane-Emden type equation:
\begin{equation}\label{lane-emden}
-\Delta u+(-\Delta)^s u=u^q\quad\text{in }\mathbb R^n,  
\end{equation}
where $n\geq 3$, $s\in(0,1)$, $q>1$, and
$$
(-\Delta)^s u(x)=C_{n,s} PV \int_{\mathbb{R}^n}\frac{u(x)-u(y)}{|x-y|^{n+2s}}dy=C_{n,s} \lim_{\varepsilon\searrow0} \int_{\mathbb{R}^n\backslash B_{\varepsilon}(x)}\frac{u(x)-u(y)}{|x-y|^{n+2s}}dy,
$$
with 
$C_{n,s}=\frac{2^{2s}\Gamma(s+n/2)}{\pi^{n/2}\Gamma(-s)}$. 

The operator $-\Delta +(-\Delta)^s $ combines classical local diffusion (Brownian motion) with nonlocal long-range jumps (a stable L\'evy process). It arises naturally in models where short-range and anomalous diffusion compete, and it provides a broader mathematical framework for analyzing such mixed processes. This type of  operator gives rise to interesting mathematical questions, notably due to the lack of scale invariance and the coupling of local and nonlocal effects---where $-\Delta$ dominates locally while $(-\Delta)^s$ prevails at long distances, as illustrated in Proposition \ref{Gamma-estimat}. In recent years, such operators have been extensively studied; see, e.g., \cite{AC2021,BI2008,BDVV2022,BDVV2023,BMV2024,BPV2025,BVDV2021,Bdd2022,DM2024,dEJ2017,DPV2022,FR2024,GK2022,GL2023,GM1992,GL1984,RS2015,SVWZ2022,SVWZ2025}.
Specific applications include control problems and stochastic games \cite{BJK2010,BK2023,DRZ2021,JK2006}, phase transition problems \cite{CS2016,CV2013},
damping in elastic systems \cite{DP2021}, population dynamics \cite{DPV2023,DV2021}.

Liouville theorems are fundamental tools in the study of partial differential equations, providing crucial a priori estimates and characterizing fundamental solution properties such as existence and uniqueness. Their significance has led to extensive research in diverse settings, including both local and nonlocal elliptic and parabolic equations. We want to study the following Liouville-type Theorem: for some elliptic operator $L$, 
\begin{equation}\label{L-super}
Lu\geq u^q\quad \text{in }\mathbb R^n, 
\end{equation}
admits no positive solutions if and only if $q<q_{L}$, where $q_{L}$ depend only on $L$ and $n$. In 1980, Gidas \cite{Gidas1980} proved that for $L=-\Delta$, \eqref{L-super} has no positive classical solutions if and only if $q<\frac{n}{n-2}$, see also \cite{BCN1994}. For $p$-Laplacian $L=-\Delta_p$ $(p>1)$ defined by $-\Delta_pu=-\mathrm{div}(|\nabla u|^{p-2}\nabla u)$,  Mitidieri and Pokhozhaev \cite{MP1998}, Birindelli and Demengel \cite{BD2002}, and Serrin and Zou \cite{SZ2002} established that  the nonexistence of positive solutions to \eqref{L-super} is equivalent to the condition $p\geq n$ or $1<p<n$ and  $q\leq \frac{n(p-1)}{n-p}$. For fractional Laplacian $L=(-\Delta)^s$ $(s\in(0,1))$, it was shown by Felmer and Quaas \cite{FQ2011} and Wang and Xiao \cite{WX2016} that there exist no positive solutions to \eqref{L-super}  if and only if $q\leq \frac{n}{n-2s}$. More generally, Del Pezzo and Quaas \cite{DQ2026} and Liu \cite{Liu2025} considered fractional $p$-Laplacian $L=(-\Delta_p)^s$ $(p>1,s\in(0,1))$, which is defined by
$$
(-\Delta_p)^s u(x)= PV \int_{\mathbb{R}^n}\frac{|u(x)-u(y)|^{p-2}(u(x)-u(y))}{|x-y|^{n+ps}}dy.
$$
They obtained that \eqref{L-super} possesses no positive solutions if and only if $sp\geq n$ or $sp<n$ and  $q\leq \frac{n(p-1)}{n-sp}$. For more general local and nonlocal operators, we refer to \cite{AS2011,BBF2026,CL2000,FQ2011,MP2001}. 

Here we also mention another widely studied Liouville-type theorem, which traces back to the celebrated work of Gidas and Spruck \cite{GS1981}. Specifically, through delicate integral identities, they proved that for $1<q<\frac{n+2}{n-2}$, the equation $-\Delta u=u^q$ in $\mathbb R^n$ admits no positive solutions. Later, Chen and Li \cite{CL1991} provided another elegant proof using the moving plane method. Their works greatly stimulated the study of such Liouville-type theorems for a wide range of more general operators (see \cite{CLL2017,CLM2020,MP2001,SZ2002}). Our paper establishes the nonexistence of positive solutions for the mixed local-nonlocal operator $-\Delta+(-\Delta)^s$ in the regime $q\leq \frac{n}{n-2s}$, where the critical exponent of the Liouville theorem associated with this operator remains unknown. See Barrios, Del Pezzo, and Quaas \cite{BDQ2025} for some discussion of the critical exponent. 

We call $u$ is a positive classical supersolution of \eqref{lane-emden}, if $u\in C^2(\mathbb R^n)\cap\mathcal{L}^{2s}(\mathbb R^n)$ satisfies
$$
-\Delta u+(-\Delta)^s u\geq u^q\quad\text{in }\mathbb R^n,  
$$
where 
$$
\mathcal{L}^{2s}(\mathbb R^n)=\left\{u\in L^1_{loc}(\mathbb R^n):\int_{\mathbb R^n}\frac{|u(x)|}{1+|x|^{n+2s}}dx<\infty\right\}.
$$
We will mainly consider the positive distributional supersolution $u$ of \eqref{lane-emden}, i.e., $u\in \mathcal{L}^{2s}(\mathbb R^n)\cap L_{loc}^q(\mathbb R^n) $ such that for any $\phi\in C^{\infty}_c(\mathbb R^n)$ with $\phi\geq 0$,
$$
\int_{\mathbb R^n}u[(-\Delta)\phi+(-\Delta)^s\phi]dx\geq \int_{\mathbb{R}^n}u^q\phi dx.
$$
Note that the above definition is well defined, since it is easy to check that $|(-\Delta)^s\phi(x)|\leq C/(1+|x|^{n+2s})$ for $\phi\in C^{\infty}_c(\mathbb R^n)$. 

Now, we state our main theorem.

\begin{theorem}\label{main-theorem-1}
If $1<q\leq \frac{n}{n-2s}$, then \eqref{lane-emden} has no positive distributional supersolutions.

If $q> \frac{n}{n-2s}$, then \eqref{lane-emden} has a smooth bounded positive classical supersolution.
\end{theorem}
\begin{remark}
Specifically, for $q>\frac{n}{n-2s}$, we define 
\begin{equation}\label{certain function}
f(x):=(1+|x|^2)^{\frac{2s-n+\sigma}{2}},\quad\text{for }x\in\mathbb R^n,
\end{equation}
where $\sigma>0$ is small so that
\begin{align*}
&\sigma+q(n-2s-\sigma)>n,\\
&2(\sigma+s)\leq n.
\end{align*}
Then, we can find a constant $\delta_0>1$ depending only on $n,s,q,\sigma$ such that for any $\delta\geq \delta_0$,
$$
f_\delta(x)=\delta^{-\frac2{q-1}} f(\frac x\delta)
$$
is a smooth bounded positive classical supersolution of \eqref{lane-emden}.
\end{remark}
\begin{remark}
   The above theorem implies the nonexistence of positive solutions for the mixed local-nonlocal Lane-Emden type    equation \eqref{lane-emden} in the regime $q\leq \frac{n}{n-2s}$. Furthermore, combining the existence result \cite{BDQ2025} with $q$ slightly smaller than $\frac{n+2}{n-2}$,  we conclude that the critical exponent for the existence of positive solutions to \eqref{lane-emden} lies in the interval $[\frac{n}{n-2s},\frac{n+2}{n-2})$. 
\end{remark}
{
The paper is organized as follows. 

In Section~\ref{Pre}, we establish sharp estimates for a fundamental solution and prove a maximum principle in the distributional sense. We then derive an equivalent integral inequality for supersolutions of \eqref{lane-emden}. In addition, we recall some properties of $(-\Delta)^s f$ where $f$ is given by \eqref{certain function}, as they play a fundamental role in the construction of the supersolutions.

In Section~\ref{proof-thm-main-1}, we prove Theorem~\ref{main-theorem-1}. In the subcritical case, we present two different methods. The first uses powers of $C^\infty_c$ functions as test functions and derives the decay of $\|u\|_{L^p(B_R)}$. The second proof employs the equivalent integral characterization and the lower bound estimate of the fundamental solution to derive a blow-up through iteration, thereby reaching a contradiction. This proof also extends to the critical case. In the supercritical case, we compare the behaviors of the Laplacian and the fractional Laplacian of $f$ at infinity. Due to the different scales of $-\Delta$ and $(-\Delta)^s$, we finally construct global supersolutions by a scaling argument.}

\section{Preliminary}\label{Pre}
\subsection{The fundamental solution }
In this subsection, we establish sharp estimates for a fundamental solution of $-\Delta +(-\Delta)^s$.

Recall that the heat kernel associated with $-\Delta +(-\Delta)^s$ is defined by \begin{equation}\label{mix-heat kernek}
H(x,t):=\int_{\mathbb R^n}e^{-t(|\xi|^2+|\xi|^{2s})+ix\cdot\xi}d\xi,\quad x\in\mathbb R^n,t>0,
\end{equation}
where $|\xi|^2+|\xi|^{2s}$ is the Fourier multiplier of $-\Delta +(-\Delta)^s.$
Based on the estimate for $H$ established by Song and Vondra\v cek \cite{SV2007}, we obtain the following sharp estimates for a fundamental solution near 0 and $\infty$. See also \cite{DSVZ2025} for further properties of $H$.

\begin{prop}\label{Gamma-estimat}
Let $\Gamma$ be defined by
\begin{equation}
\Gamma(x)=\int^{\infty}_{0}H(x,t)dt,
\end{equation}
where $H$ is given by \eqref{mix-heat kernek}. Then $\Gamma$ is nonnegative, radially symmetric, and nonincreasing as a function of $r=|x|$, and  there exists a positive constant $C$ such that
\begin{align}
\frac1C\frac{1}{|x|^{n-2}}\leq &\Gamma(x)\leq C\frac{1}{|x|^{n-2}},\quad\text{for }0<|x|< 1,\label{mix-estimate-zero}\\
\frac1C\frac{1}{|x|^{n-2s}}\leq &\Gamma(x)\leq C\frac{1}{|x|^{n-2s}},\quad\text{for }|x|\geq 1.\label{mix-estimate-infty}
\end{align}
Moreover, $\Gamma$ is a fundamental solution of the operator $-\Delta +(-\Delta)^s$, i.e., $\Gamma$ satisfies 
\begin{equation}\label{fundame-equ}
[-\Delta +(-\Delta)^s]\Gamma=\delta_0\quad\text{in }\mathbb R^n,
\end{equation}
in the distribution sense, where $\delta_0$ denotes the Dirac mass at point 0.
\end{prop}
\begin{proof}
Step 1. We collect some useful estimates.

We denote 
$$p^{(2)}(x,t)=\frac1{(4\pi t)^{\frac n2}}\exp(-\frac{|x|^2}{4t})$$
and $p^{(2s)}(x,t)$ are the heat kernels associated with $-\Delta$ and $(-\Delta)^s$, respectively. We also define
\begin{align*}
\tilde{p}^{(2)}(x,t)=\frac1{(4\pi t)^{\frac n2}}\exp(-\frac{|x|^2}{16t}),\quad\hat{p}^{(2)}(x,t)=\frac1{(4\pi t)^{\frac n2}}\exp(-\frac{|x|^2}{t}),
\end{align*}
and
\begin{align*}
q_1(x,t)&=\begin{cases}
\hat{p}^{(2)}(x,t),\quad\text{for }|x|^2<t<|x|^{2s}\leq 1,\\
\max\{\hat{p}^{(2)}(x,t),{p}^{(2s)}(x,t)\},\quad\text{for }t<|x|^{2}\leq 1,\\
{p}^{(2)}(x,t),\quad\text{for }|x|^{2s}\leq t\leq 1,\\
{p}^{(2s)}(x,t),\quad\text{for } t\geq 1\text{ or }|x|\geq 1,
\end{cases}\\
q_2(x,t)&=\begin{cases}
p^{(2)}(x,t),\quad\text{for }|x|^2<t<|x|^{2s}\leq 1,\\
\max\{\tilde{p}^{(2)}(x,t),{p}^{(2s)}(x,t)\},\quad\text{for }t<|x|^{2}\leq 1,\\
{p}^{(2)}(x,t),\quad\text{for }|x|^{2s}\leq t\leq 1,\\
{p}^{(2s)}(x,t),\quad\text{for } t\geq 1\text{ or }|x|\geq 1.
\end{cases}
\end{align*}
Theorem 2.13 in \cite{SV2007} shows that there exists a positive constant $C_1$ such that for any $(x,t)\in\mathbb R^n\times (0,+\infty)$,
\begin{equation}\label{mix-kernel-estimate}
 \frac{1}{C_1}q_1(x,t)\leq H(x,t)\le C_1 q_2(x,t).
\end{equation}
Note that by Blumenthal and Getoor \cite{BG1960}, there exists a positive constant $C_2$ such that for any $(x,t)\in\mathbb R^n\times(0,+\infty)$,
\begin{equation}\label{s-heat kernel-estimate}
\frac1{C_2}\min\{t^{-\frac n{2s}},t|x|^{-n-2s}\}\leq p^{(2s)}(x,t)\leq C_2\min\{t^{-\frac n{2s}},t|x|^{-n-2s}\}.
\end{equation}
We observe that 
\begin{equation}\label{min-cases}
\min\{t^{-\frac n{2s}},t|x|^{-n-2s}\}=
\begin{cases}
t|x|^{-n-2s},\quad\text{if }|x|^{2s}>t,\\
t^{-\frac n{2s}},\quad\text{if }|x|^{2s}\leq t.
\end{cases}
\end{equation}

Step 2. We prove \eqref{mix-estimate-zero}.

Assume that $0<|x|\leq 1$, then
\begin{align*}
\Gamma(x)&\leq C_1\int^\infty_0q_2(x,t)dt\\
&= C _1\left(\int^{|x|^2}_0 \max\{\tilde{p}^{(2)}(x,t),{p}^{(2s)}(x,t)\}dt+\int^{1}_{|x|^{2}}{p}^{(2)}(x,t)dt+\int^{\infty}_1 p^{(2s)}(x,t)dt\right)
\end{align*}
First, note that $t<|x|^2$ implies that $t<|x|^{2s}$, hence by \eqref{s-heat kernel-estimate}-\eqref{min-cases}, we have
$$
 \int^{|x|^2}_0 {p}^{(2s)}(x,t)dt\leq C \int^{|x|^2}_0 t|x|^{-n-2s}dt=C|x|^{4-n-2s}.
$$
Next, by the variable substitution $z=|x|^2/t$, we obtain
\begin{align*}
 \int^{|x|^2}_0 \tilde{p}^{(2)}(x,t)dt= C\int^{|x|^2}_0 t^{-\frac n2}e^{-\frac{|x|^2}{16t}} dt=C\int^{\infty}_1|x|^{2-n}z^{\frac n2-2}e^{-\frac z{16}}dz=C|x|^{2-n},
\end{align*}
and
\begin{align*}
 \int^{1}_{|x|^{2}}{p}^{(2)}(x,t)dt= C \int^{1}_{|x|^{2}}t^{-\frac n2}e^{-\frac{|x|^2}{4t}} dt = C\int^{1}_{|x|^2}|x|^{2-n}z^{\frac n2-2}e^{-\frac z{4}}dz\leq C |x|^{2-n}.
\end{align*}
In addition, \eqref{s-heat kernel-estimate} and \eqref{min-cases} yield that
$$
\int^{\infty}_1 p^{(2s)}(x,t)dt \leq C\int^{\infty}_1 t^{-\frac n{2s}}dt\leq C.
$$
Combining the above facts, we get $\Gamma(x)\leq C|x|^{2-n}$ for all $0<|x|\leq 1$. The argument for proving $\Gamma(x)\geq C'|x|^{2-n}$ for all $0<|x|\leq 1$ is analogous.

Step 3. We prove \eqref{mix-estimate-infty}.

Assume that $|x|\geq 1$, then by \eqref{min-cases}, we have
\begin{align*}
\Gamma(x)&\leq C_1\int^\infty_0q_2(x,t)dt= C_1\int^\infty_0p^{(2s)}(x,t)dt\leq C\left(\int_0^{|x|^{2s}}t|x|^{-n-2s}dt+\int^{\infty}_{|x|^{2s}}t^{-\frac n{2s}}dt\right)\\
&\leq C|x|^{2s-n}.
\end{align*}
Similar discussion yields $\Gamma(x)\geq C'|x|^{2s-n}$ for all $|x|\geq 1$.

Now, combining \eqref{mix-estimate-zero} and \eqref{mix-estimate-infty},  $\Gamma$ is well defined in $\mathbb R^n\backslash\{0\}$. Then, by Theorem 3.1 in \cite{DSVZ2025}, we know that $H$ is nonnegative, radially symmetric, and nonincreasing with respect to  $r=|x|$, hence so is $\Gamma$.

Step 4. We verify  that $\Gamma$ is a fundamental solution of $-\Delta +(-\Delta)^s.$
Taking the Fourier transform, equation \eqref{fundame-equ} becomes equivalent to $$
\mathcal F(\Gamma)(\xi)=\frac{1}{|\xi|^2+|\xi|^{2s}},
$$  where $\mathcal F(\cdot)$ denotes the Fourier transform. 
Hence, it suffices to show that
for any $\phi\in C^\infty_c(\mathbb{R}^n)$, it holds
\begin{equation}\label{equ-dis-test}
\langle \mathcal{F}^{-1}\left(\frac{1}{|\xi|^2+|\xi|^{2s}}\right),\phi\rangle= \langle\Gamma,\phi\rangle.
\end{equation}

First, by \eqref{mix-estimate-zero}, we have $\Gamma\phi\in L^1(\mathbb R^n)$. Next, for a fixed $t>0$, we observe that $A(x,\xi):=e^{-t(|\xi|^2+|\xi|^{2s})+ix\cdot\xi}\phi(x)\in L^1(\mathbb R^n\times\mathbb R^n)$. Moreover, since $\mathcal{F}^{-1}(\phi)$ is a Schwartz function, then $|\mathcal{F}^{-1}(\phi)(\xi)|(1+|\xi|)^{N}$ is bounded for any $N\in\mathbb N$ and thus  $\mathcal{F}^{-1}(\phi)(\xi)\frac{1}{|\xi|^2+|\xi|^{2s}}\in L^1(\mathbb R^n)$. With these facts, employing \eqref{mix-heat kernek} and applying Fubini's theorem repeatedly, we derive
\begin{align*}
\langle\Gamma,\phi\rangle&=\int_{\mathbb R^n}\int^{\infty}_{0}H(x,t)\phi(x) dtdx\\
&=\int^{\infty}_{0}\int_{\mathbb R^n}H(x,t)\phi(x) dxdt\\
&=\int^{\infty}_{0}\int_{\mathbb R^n}\int_{\mathbb R^n}e^{-t(|\xi|^2+|\xi|^{2s})+ix\cdot\xi}\phi(x)d\xi dx dt\\
&=\int^{\infty}_{0}\int_{\mathbb R^n}e^{-t(|\xi|^2+|\xi|^{2s})}\left(\int_{\mathbb R^n}e^{ix\cdot\xi}\phi(x)dx\right) d\xi  dt\\
&=\int^{\infty}_{0}\int_{\mathbb R^n}e^{-t(|\xi|^2+|\xi|^{2s})}\mathcal{F}^{-1}(\phi)(\xi) d\xi  dt\\
&=\int_{\mathbb R^n}\mathcal{F}^{-1}(\phi)(\xi)\left(\int^{\infty}_{0}e^{-t(|\xi|^2+|\xi|^{2s})} dt\right)d\xi\\
&=\int_{\mathbb R^n}\mathcal{F}^{-1}(\phi)(\xi)\frac{1}{|\xi|^2+|\xi|^{2s}}d\xi=\langle \mathcal{F}^{-1}\left(\frac{1}{|\xi|^2+|\xi|^{2s}}\right),\phi\rangle.
\end{align*}
Therefore, we obtain \eqref{equ-dis-test} and  complete the proof.
\end{proof}
\subsection{A maximum principle in the distribution sense}
Next, we prove the following maximum principle for the mixed local-nonlocal operator $-\Delta +(-\Delta)^s $ in $\mathbb R^n$ in the distribution sense. For the maximum principles (for both weak and classical solutions), see \cite{BDVV2022}. 

\begin{lemma}\label{maximum principle}
Let $u\in \mathcal{L}^{2s}(\mathbb R^n)$ satisfy
\begin{equation}\label{mix-superharmonic}
-\Delta u+(-\Delta)^s u\geq 0\quad \text{in }\mathbb R^n
\end{equation}
in the distribution sense, i.e., for any $\phi\in C^{\infty}_c(\mathbb R^n)$ with $\phi\geq 0$,
$$
\int_{\mathbb R^n}u[(-\Delta)+(-\Delta)^s]\phi dx\geq 0.
$$
Assume that 
\begin{equation}\label{liminf-geq0}
\liminf_{|x|\rightarrow\infty}u(x)\geq 0.
\end{equation}
Then, $u\geq 0$ a.e. in $\mathbb R^n$.
\end{lemma}
\begin{proof}
 We use the smooth approximation. Fix  $\eta\in C^\infty_c(B_1)$ is radial about 0 satisfying  $\eta\geq 0$ and $\int_{\mathbb R^n}\eta dx=1$. 
 For $\varepsilon\in (0,1)$ and a function $v$, we denote $v_{\varepsilon}(x)=\int_{\mathbb R^n}v(x-y)\varepsilon^{-n}\eta(\varepsilon^{-1}y)dy$. 
 
 Now, We claim that $u_\varepsilon\in C^{\infty}(\mathbb R^n) $ satisfies the following properties: 
 
 $\mathrm{(i)}$ $u_\varepsilon\in \mathcal{L}^{2s}(\mathbb R^n)$. Note that for any $y\in B_{\varepsilon}$ and $x\in\mathbb R^n$, we have
     $$
\frac{1+|x-y|^{n+2s}}{1+|x|^{n+2s}}\leq C,
     $$
hence
\begin{align*}
\int_{\mathbb R^n}\frac{|u_\varepsilon(x)|}{1+|x|^{n+2s}}dx&=\int_{\mathbb R^n}\int_{\mathbb R^n}\frac{|u(x-y)|}{1+|x|^{n+2s}}dx\varepsilon^{-n}\eta(\varepsilon^{-1}y)dy\\
&\leq C\int_{\mathbb R^n}\int_{\mathbb R^n}\frac{|u(x-y)|}{1+|x-y|^{n+2s}}dx\varepsilon^{-n}\eta(\varepsilon^{-1}y)dy\\&=C\int_{\mathbb R^n}\frac{|u(x)|}{1+|x|^{n+2s}}dx<\infty.
\end{align*}

$\mathrm{(ii)}$ $-\Delta u_{\varepsilon}+(-\Delta)^s u_{\varepsilon}\geq 0$ in $\mathbb R^n$ in the classical sense. Take any $\phi\in C^{\infty}_c(\mathbb R^n)$ with $\phi\geq 0$. Note that 
$$
(-\Delta)^s\phi(x)=\frac{C_{n,s}}{2}\int_{\mathbb R^n}\frac{2\phi(x)-\phi(x+y)-\phi(x-y)}{|y|^{n+2s}}dy,
$$
we easily verify that $((-\Delta)^s\phi)_\varepsilon=(-\Delta)^s\phi_\varepsilon$. We also have $(-\Delta\phi)_\varepsilon=-\Delta\phi_\varepsilon$. Hence, 
\begin{align*}
\int_{\mathbb R^n}[-\Delta+(-\Delta)^s]u_\varepsilon\phi dx&=\int_{\mathbb R^n}u_\varepsilon[-\Delta+(-\Delta)^s]\phi dx=\int_{\mathbb R^n}u([-\Delta+(-\Delta)^s]\phi)_\varepsilon dx\\&=\int_{\mathbb R^n}u[-\Delta+(-\Delta)^s]\phi_\varepsilon dx\geq 0,
\end{align*}
where the last inequality follows from \eqref{mix-superharmonic}.
By the arbitrariness of $\phi$, we obtain $-\Delta u_{\varepsilon}+(-\Delta)^s u_{\varepsilon}\geq 0$ pointwise in $\mathbb R^n$.

$\mathrm{(iii)}$ $\liminf_{|x|\rightarrow\infty}u_\varepsilon(x)\geq 0.$ By \eqref{liminf-geq0}, for any $\delta>0$, there exists some $R>0$ such that $u(x)\geq -\delta$ for a.e. $x\in B_R^c$. Note that for any $x\in B_{R+\varepsilon}^c$ and $y\in B_\varepsilon$,  we have $|x-y|\geq R$, hence
$$
u_{\varepsilon}(x)=\int_{B_\varepsilon}u(x-y)\varepsilon^{-n}\eta(\varepsilon^{-1}y)dy\geq -\delta\int_{B_\varepsilon}\varepsilon^{-n}\eta(\varepsilon^{-1}y)dy=-\delta.
$$
Therefore, we obtain that $\liminf_{|x|\rightarrow\infty}u_\varepsilon(x)\geq 0.$

 From the above properties of $u_{\varepsilon}$, we have $u_\varepsilon\geq 0$ in $\mathbb R^n$. Otherwise, $\inf_{\mathbb R^n}u_\varepsilon<0$. By $\mathrm{(iii)}$, there exists an $x_0\in\mathbb R^n$ such that $u_\varepsilon(x_0)=\min_{\mathbb R^n}u_{\varepsilon}<0$. However, we have
$$
[-\Delta +(-\Delta)^s]u_\varepsilon(x_0)\leq (-\Delta)^s u_\varepsilon(x_0)=\int_{\mathbb R^n}\frac{ u_\varepsilon(x_0)- u_\varepsilon(y)}{|x_0-y|^{n+2s}}dy<0,
$$
which contradicts $\mathrm{(ii)}$. Therefore, we obtain $u_\varepsilon\geq 0$ in $\mathbb R^n$. Let $\varepsilon\rightarrow 0$, we get  $u\geq 0$ a.e. in $\mathbb R^n$.
\end{proof}
\subsection{An equivalent integral inequality for supersolutions 
}
Based on the fundamental solution and the maximum principle above, we establish an equivalent integral inequality for  supersolutions of \eqref{lane-emden}.
\begin{prop}\label{equiv-ineq}
Let $u$ be a nonnegative distributional supersolution of \eqref{lane-emden}. Then
\begin{equation}\label{integral-equ}
u(x)\geq \int_{\mathbb R^n}\Gamma(x-y)u^q(y)dy\quad\text{a.e. in }\mathbb R^n.
\end{equation}
\end{prop}

\begin{proof}
For $R>0$, denote 
$$
u_R=  
\begin{cases}
u\quad\text{in }  B_R,\\
0\quad\text{in }  \mathbb R^n\backslash B_R.
\end{cases}
$$
Let 
$$
w_R(x)=\int_{\mathbb R^n} \Gamma(x-y)u^q_R(y)dy.
$$
Then by Proposition \ref{Gamma-estimat}, one can easily check that $w_R\in\mathcal{L}^{2s}(\mathbb{R}^n)$ is a distributional solution to
$$
[-\Delta +(-\Delta)^s]w_R=u_R^q\quad\text{in }\mathbb R^n.
$$
with the asymptotic behavior\begin{equation}\label{wR-decay}
w_R(x)\rightarrow 0,\ \mbox{as}\ |x|\rightarrow\infty.
\end{equation}
Set $v_R=u-w_R$, then $v_R$ is a distributional solution of 
\begin{equation}\label{vR-equation}
[-\Delta +(-\Delta)^s]v_R=u^q-u_R^q\geq 0\quad\text{in }\mathbb R^n,
\end{equation}
with $\liminf_{|x|\rightarrow\infty}v_R(x)\geq 0$.
Applying Lemma \ref{maximum principle} to $v_R$, we have $v_R\geq 0$ a.e. in $\mathbb R^n$. 
Letting $R\rightarrow+\infty$ and applying the monotone convergence theorem, we obtain \eqref{integral-equ}.
\end{proof}
\begin{remark}
As a consequence of Proposition \ref{equiv-ineq}, if $u$ is a nonnegative distributional supersolution of \eqref{lane-emden} that vanishes in a set with positive measure, then $u=0$ a.e. in $\mathbb{R}^n$.
\end{remark}

\subsection{The fractional Laplacian of \eqref{certain function}}
To construct a supersolution in the supercritical case, we need the fractional Laplacian of a certain function. Specifically, for some constant $\sigma>0$, we define 
\begin{equation}\label{special function}
f(x):=(1+|x|^2)^{\frac{2s-n+\sigma}{2}},\quad\text{for }x\in\mathbb R^n.
\end{equation}
For $q>\frac{n}{n-2s}$, we can choose $\sigma>0$ small so that
\begin{align}
&\sigma+q(n-2s-\sigma)>n,\label{sigma-choose-1}\\
&2(\sigma+s)\leq n.\label{sigma-choose-2}
\end{align}

By Lemmas 4.2-4.3 in Liu \cite{Liu2025}, we have the following useful lemma.
\begin{lemma}\label{special function-estimate}
Let $f$ be defined by \eqref{special function} with  $q>\frac{n}{n-2s}$ and \eqref{sigma-choose-1}-\eqref{sigma-choose-2}. Then $(-\Delta)^s f$ is a continuous positive function in $\mathbb R^n$. Moreover, for any $r>0$, let $g_r(z)$ be defined by 
\begin{equation}\label{gr-defin}
g_r(z):=\left(\frac1{r^2}+|z|^2\right)^{\frac{2s-n+\sigma}{2}},\quad\text{for }z\in\mathbb R^n. 
\end{equation}
Then 
\begin{equation}
    \lim_{r\rightarrow\infty}(-\Delta)^s g_r(e_1)=B,
\end{equation}
where $e_1=(1,0,\dots,0)\in\mathbb R^n$ and $B$ is a positive constant.
\end{lemma}

\section{Proof of Theorem \ref{main-theorem-1}}\label{proof-thm-main-1}
In this section, we prove Theorem \ref{main-theorem-1}.
\begin{proof}[Proof of Theorem \ref{main-theorem-1}]
Case 1: $1<q< \frac n{n-2s}$. 

We present two different proofs.

\textbf{Proof by test functions:} Let $u$ be a positive supersolution of \eqref{lane-emden} and $\eta\in C^\infty_c(B_2)$ satisfy $0\leq \eta\leq 1$, $\eta|_{B_1}=1$. For any $R>0$, we set $\eta_{R}=\eta(\frac\cdot R)$  For some $m\in\mathbb Z_{+}$ to be determined, we use $\eta^m_{R}$ as a test function for \eqref{lane-emden}. Note that 
$$
|\nabla \eta_R|\leq \frac CR,\quad |\nabla^2\eta_R|\leq \frac C{R^2}.
$$
Hence, we have 
\begin{equation*}
\int_{\mathbb R^n}u(-\Delta\eta^m_R)dx=-\int_{\mathbb R^n}u(m\eta^{m-1}_R\Delta\eta_R+m(m-1)\eta^{m-2}_R|\nabla\eta_R|^2)dx\leq \frac{C}{R^2}\int_{\mathbb R^n}u\eta^{m-1}_Rdx.
\end{equation*}
Now we set $m=\lfloor\frac{q}{q-1}\rfloor+1$, then $m-1-\frac{m}{q}>0$. By H\"older inequality, we obtain that  
\begin{equation}\label{Delta-eta}
\begin{aligned}
\int_{\mathbb R^n}u(-\Delta\eta^m_R)dx&\leq \frac{C}{R^2}\int_{\mathbb R^n}u\eta^{\frac m q}_R\eta^{m-1-\frac{m}{q}}_Rdx\\
&\leq \frac{C}{R^2}\left(\int_{\mathbb R^n}u^q\eta^m_Rdx\right)^{\frac1q}\left(\int_{\mathbb R^n}\eta^{(m-1-\frac{m}{q})(\frac q{q-1})}_Rdx\right)^{1-\frac1q}\\
&\leq CR^{n(1-\frac1q)-2}\left(\int_{\mathbb R^n}u^q\eta^m_Rdx\right)^{\frac1q}.
\end{aligned}
\end{equation}
Next, by the convexity of $t\mapsto t^m$, 
we have $$(a^m-b^m)\leq ma^{m-1}(a-b),\ \mbox{for all}\ a,b\geq 0,$$ hence
\begin{equation*}
\begin{aligned}
\int_{\mathbb R^n}u(-\Delta)^s(\eta^m_R)dx&=\int_{\mathbb R^n}\int_{\mathbb R^n}u(x)\frac{\eta^m_R(x)-\eta^m_R(y)}{|x-y|^{n+2s}}dydx\\
&\leq \int_{\mathbb R^n}\int_{\mathbb R^n}u(x)\frac{m\eta^{m-1}_R(x)(\eta_R(x)-\eta_R(y))}{|x-y|^{n+2s}}dydx\\
&= \int_{\mathbb R^n}u(x)m\eta^{m-1}_R(x)\frac{1}{R^{2s}}\int_{\mathbb R^n}\frac{(\eta(\frac x R)-\eta(\frac y R))}{|\frac x R-\frac y R|^{n+2s}}\frac{dy}{R^n}dx\\
&= \frac{1}{R^{2s}}\int_{\mathbb R^n}u(x)m\eta^{m-1}_R(x)[(-\Delta)^s\eta]\left(\frac x R\right)dx.
\end{aligned}
\end{equation*}
Note that $supp(\eta_R)\subset B_{2R}$ and $|(-\Delta)^s\eta|\leq C$ in $B_2$ yield that
$$
\int_{\mathbb R^n}u(-\Delta)^s(\eta^m_R)dx\leq  \frac{C}{R^{2s}}\int_{\mathbb R^n}u\eta^{m-1}_Rdx.
$$
Following a discussion similar to \eqref{Delta-eta}, we obtain 
\begin{equation}\label{s-Delta-eta}
 \int_{\mathbb R^n}u(-\Delta)^s(\eta^m_R)dx\leq    CR^{n(1-\frac1q)-2s}\left(\int_{\mathbb R^n}u^q\eta^m_Rdx\right)^{\frac1q}. 
\end{equation}
On the other hand, by $u$ is a supersolution of \eqref{lane-emden}, we get
\begin{equation}\label{mix-super}
\int_{\mathbb R^n}u[(-\Delta)+(-\Delta)^s](\eta^m_R)dx \geq \int_{\mathbb R^n}u^p\eta^m_R dx.    
\end{equation}
Combining \eqref{Delta-eta}-\eqref{mix-super}, we have
\begin{equation}\label{decay-estimate-inte}
\left(\int_{\mathbb R^n}u^q\eta^m_Rdx\right)^{1-\frac1q}\leq C(R^{n(1-\frac1q)-2s}+R^{n(1-\frac1q)-2}).
\end{equation}
By $q<\frac n{n-2s}$, we have 
$$
n(1-\frac1q)-2<n(1-\frac1q)-2s<0.
$$
Hence, let $R\rightarrow+\infty$ in \eqref{decay-estimate-inte}, we get $u\equiv0$ in $\mathbb R^n$, which contradicts $u>0$ in $\mathbb R^n$.

\textbf{Proof by equivalent integral characterization:}  Let $u$ be a positive supersolution of \eqref{lane-emden}. By \eqref{integral-equ} and \eqref{mix-estimate-infty}, we have, for $|x|\geq 2$,
$$
u(x)\geq \int_{B_1}\Gamma(x-y)u^q(y)dy\geq\frac1{C}\int_{B_1}\frac1{|x-y|^{n-2s}}u^q(y)dy\geq \frac{C_1}{|x|^{n-2s}}. 
$$
for some constant $C_1>0$.
Hence, for $|x|\geq 4$, 
\begin{align*}
u(x)&\geq \int_{B_{\frac{|x|}{2}}(x)\backslash B_{\frac{|x|}{4}}(x)}\Gamma(x-y)u^q(y)dy\geq \frac{1}{C}\int_{B_{\frac{|x|}{2}}(x)\backslash B_{\frac{|x|}{4}}(x)}\frac{1}{|x-y|^{n-2s}}\frac{C^q_1}{|y|^{q(n-2s)}}dy\\
&\geq\frac{C_2}{|x|^{q(n-2s)-2s}},\dots
\end{align*}
By induction, for $j=1,2,\dots$, when  $|x|\geq 2^j$, we obtain
\begin{equation}\label{decay-alphaj}
u(x)\geq \frac{C_j}{|x|^{\alpha_j}},
\end{equation}
where $C_j>0$ and
$$
\alpha_1=n-2s,\quad\alpha_j=q\alpha_{j-1}-2s~(j\geq 2).
$$
Note that $1<q<\frac n{n-2s}$ implies $n-2s+\frac{2s}{1-q}<0$, hence, for $j$ large, 
$$
\alpha_j=q^{j-1}\alpha_1-2s(1+q+q^2+\cdots +q^{j-1})=\left(n-2s+\frac{2s}{1-q}\right)q^{j-1}-\frac{2s}{1-q}<0.
$$
We choose a $j_0$ such that $\alpha_{j_0}<0$. Then, by \eqref{integral-equ} and \eqref{decay-alphaj}, we get, for $|x|\leq 2^{j_0-1}$,
$$
u(x)\geq \int_{\mathbb R^n\backslash B_{2^{j_0}}}\Gamma (x-y)u^q(y)dy\geq  \frac{1}{C}\int_{\mathbb R^n\backslash B_{2^{j_0}}}\frac1{|x-y|^{n-2s}}\frac{C^q_{j_0}}{|y|^{q\alpha_{j_0}}}dy=+\infty.
$$
Therefore we obtain a contradiction, which shows that \eqref{lane-emden} has no positive supersolution.

Case 2: $q=\frac n{n-2s}$.

Let $R\geq 1$. By \eqref{integral-equ}, \eqref{mix-estimate-infty}, and \eqref{mix-estimate-zero}, we have
\begin{equation}\label{u-inf-estimate}
\begin{aligned}
u(x)&\geq \int_{B_R}\Gamma (x-y)u^q(y)dy\\
&=\left(\int_{B_R\cap\{|x-y|\leq 1\}}+\int_{B_R\cap\{|x-y|\geq 1\}}\right)\Gamma (x-y)u^q(y)dy\\
&\geq\frac{1}{C}\left[\int_{B_R\cap\{|x-y|\leq 1\}}\frac{u^q(y)}{|x-y|^{n-2}}dy+\int_{B_R\cap\{|x-y|\geq 1\}}\frac{u^q(y)}{|x-y|^{n-2s}}dy\right] \\
&\geq \frac{1}{C}\left[\int_{B_R\cap\{|x-y|\leq 1\}}u^q(y)dy+\frac{1}{(|x|+R)^{n-2s}}\int_{B_R\cap\{|x-y|\geq 1\}}{u^q(y)}dy\right]\\
&\geq \frac1{C}\frac{1}{(|x|+R)^{n-2s}}\int_{B_R}u^q(y)dy.
\end{aligned}
\end{equation}
Hence, 
$$
\int_{B_R}u^q(x)dx\geq \frac{1}{C^q}\int_{B_R}\frac1{(|x|+R)^{n}}dx\left(\int_{B_R}u^q(y)dy\right)^q\geq c\left(\int_{B_R}u^q(y)dy\right)^q,
$$
where $c$ is independent of $R$. Let $R\rightarrow +\infty$, we get $u\in L^q(\mathbb R^n)$. Furthermore, by \eqref{u-inf-estimate}, we have
$$
\int_{B_{2R}\backslash B_R} u^q(x)dx\geq \frac{1}{C^q}\int_{B_{2R}\backslash B_R}\frac1{(|x|+R)^{n}}dx\left(\int_{B_R}u^q(y)dy\right)^q\geq c'\left(\int_{B_R}u^q(y)dy\right)^q,
$$
where $c'$ is independent of $R$. Let $R\rightarrow +\infty$, we obtain $\int_{\mathbb{R}^n}u^q(y)dy=0$ and hence $u\equiv0$. This contradicts the fact that $u>0$.

Case 3: $q>\frac n{n-2s}$.

We will construct a positive supersolution of \eqref{lane-emden}. Let $f$ be defined by \eqref{special function} with \eqref{sigma-choose-1}-\eqref{sigma-choose-2}. We first compute $(-\Delta)^s f$. Fix $x,y\in\mathbb R^n$ with $x\neq0$. We set $r=|x|$ and write $x=rA_x e_1$ for some rotation $A_x$ in $\mathbb R^n$. We also write $y=r A_x z$ for some $z\in\mathbb R^n$. Let $g_r$ be defined by \eqref{gr-defin}, we have
$$
f(x)-f(y)=r^{2s-n+\sigma}(g_r(e_1)-g_r(z)).
$$
Hence, we obtain
\begin{equation}\label{s-delta-f}
\begin{aligned}
(-\Delta)^s f(x)&=\mathrm{p.v.}\int_{\mathbb R^n}\frac{f(x)-f(y)}{|x-y|^{n+2s}}dy=  \mathrm{p.v.}\int_{\mathbb R^n}\frac{r^{2s-n+\sigma}(g_r(e_1)-g_r(z))}{r^{n+2s}|e_1-z|^{n+2s}}r^ndz\\
&=r^{\sigma-n}(-\Delta)^s g_r(e_1)=f(x)^q\frac{r^{\sigma-n}}{(1+r^2)^{\frac{(2s-n+\sigma)q}{2}}}(-\Delta)^s g_r(e_1).
\end{aligned}
\end{equation}
On the other hand, a direct computation yields 
\begin{equation}\label{delta-f}
\begin{aligned}
(-\Delta)f(x)&=-(1+r^2)^{\frac{2s-n+\sigma}{2}-1}(2s-n+\sigma)\left((2s-n+\sigma-2)\frac{|x|^2}{1+|x|^2}+n\right)\\
&=f(x)^q\frac{r^{\sigma-n}}{(1+r^2)^{\frac{(2s-n+\sigma)q}{2}}}h(r),
\end{aligned}
\end{equation}
where 
$$
h(r):=-\frac{(1+r^2)^{\frac{2s-n+\sigma}{2}-1}}{r^{2s-n+\sigma-2}}\cdot r^{2s-2}(2s-n+\sigma)\left((2s-n+\sigma-2)\frac{r^2}{1+r^2}+n\right).
$$
Note that 
\begin{equation}\label{h-infty-0}
\lim_{r\rightarrow+\infty}h(r)=0.
\end{equation}
By \eqref{sigma-choose-1}, we have
\begin{equation}\label{com-infty}
\lim_{r\rightarrow\infty}\frac{r^{\sigma-n}}{(1+r^2)^{\frac{(2s-n+\sigma)q}{2}}}=+\infty.
\end{equation}
Combining Lemma \ref{special function-estimate} and  \eqref{s-delta-f}-\eqref{com-infty}, we can find a $R>0$ such that for any $|x|\geq R$,
\begin{equation}\label{exter-equation}
[-\Delta +(-\Delta)^s] f(x)\geq f(x)^q.
\end{equation}
Furthermore, for $\delta>0$ and $u$  a positive supersolution of \eqref{lane-emden}, we consider 
$$
u_\delta(x)=\delta^\frac{2}{q-1}u(\delta x),\quad x\in\mathbb R^n.
$$
Then $u_\delta$ satisfies
\begin{equation}\label{u_delta-equation}
-\Delta u_\delta+\delta^{2(1-s)}(-\Delta)^su_\delta\geq u_\delta^q\quad\text{in }\mathbb R^n. 
\end{equation}
Hence, to construct a positive solution of \eqref{lane-emden}, we only need to construct a positive supersolution of \eqref{u_delta-equation} for some $\delta>0$. By Lemma \ref{special function-estimate}, we know that $(-\Delta)^s f$ is a continuous positive function in $\mathbb R^n$. Now, we take 
$$
\delta=\max\left\{1,\max_{|x|\leq R}\left(\frac{(f^q+\Delta f)_+}{(-\Delta)^sf}\right)^{\frac{1}{2(1-s)}}\right\}, 
$$
then $-\Delta f+\delta^{2(1-s)}(-\Delta)^sf\geq f^q$ in $\{|x|\leq R\}$. Combining this with \eqref{exter-equation} shows that $u_\delta=f$ is a supersolution of \eqref{u_delta-equation}.  Consequently, $\delta^{-\frac2{q-1}} f( \cdot/\delta)$ is a supersolution of \eqref{lane-emden}, which completes the proof.  
\end{proof}

{\bf{Acknowledgments.}} 
Guo is partially supported by the National Natural Science Foundation of China (Grant No. 12501145), the Natural Science Foundation of Shanghai (No. 25ZR1402207),   the China Postdoctoral Science Foundation (No. 2025T180838 and No. 2025M773061), the Postdoctoral Fellowship Program of CPSF (No. GZC20252004), and the Institute of Modern Analysis-A Frontier Research Center of Shanghai.
Li and Xie are partially supported by the National Natural Science Foundation of China (Grant No. W2531006, 12250710674 and 12031012) and the Institute of Modern Analysis-A Frontier Research Center of Shanghai.

\medskip
{\bf{Date availability statement:}} No data was used for the research described in the article.

\medskip
{\bf{Conflict of interest statement:}} On behalf of all authors, the corresponding author states that there is no conflict of interest.


\end{document}